\documentclass{conm-p-l}

\usepackage{amsthm,amsfonts, cite}
\usepackage{amssymb,graphicx,color}
\usepackage[all]{xy}
\usepackage{mathpazo, dsfont}
\usepackage{eucal}
\usepackage{enumerate}
\usepackage{hyperref}
\numberwithin{equation}{section}

\theoremstyle{plain}
\newtheorem{teo}{Theorem}[section]
\newtheorem*{teo*}{Theorem}
\newtheorem{teoA}{Theorem}

\newcommand{\R}{\ensuremath{{\mathbb{R}}}}
\newcommand{\g}{\ensuremath{\mathtt{g}}}
\newcommand{\hg}{\ensuremath{\widehat{\mathtt{g}}}}
\newcommand{\tg}{\ensuremath{\widetilde{\mathtt{g}}}}
\renewcommand\vec[1]{\boldsymbol{#1}}

\begin{document}

\title[Conformal flatness of compact three-dimensional Cotton-parallel manifolds]{Conformal flatness of compact \\ three-dimensional Cotton-parallel manifolds}

\author[Ivo Terek]{Ivo Terek}
\address[Ivo Terek]{Department of Mathematics, The Ohio State University, Columbus, OH 43210, USA}
\email{terekcouto.1@osu.edu}


\begin{abstract}
  A three-dimensional pseudo-Riemannian manifold is called essentially conformally symmetric (ECS) if its Cotton tensor is parallel but nowhere-vanishing. In this note we prove that three-dimensional ECS manifolds must be noncompact or, equivalently, that every compact three-di\-men\-si\-o\-nal Cotton-parallel pseudo-Riemannian manifold must be conformally flat.
\end{abstract}
\maketitle

\section{Introduction and main result}

Pseudo-Riemannian manifolds of dimensions $n\geq 4$ whose Weyl tensor is parallel are called \emph{conformally symmetric} \cite{Chaki-Gupta}. Those which are not locally symmetric or conformally flat are called \emph{essentially conformally symmetric} (\emph{ECS}, in short).

It has been shown by Roter in \cite[Corollary 3]{Roter74} that ECS manifolds do exist in all dimensions $n\geq 4$, and in \cite[Theorem 2]{TensorNS77} that they necessarily have indefinite metric signature. The local isometry types of ECS manifolds were described by Derdzinski and Roter in \cite{local_structure}. Compact ECS manifolds exist in all dimensions $n\geq 5$ and realize all indefinite metric signatures -- see \cite{DT_1} and \cite{AGAG10}. It is not currently known if compact four-dimensional ECS manifolds exist.

When the dimension of $M$ is $n \leq 3$, the Weyl tensor vanishes and this discussion becomes meaningless. In dimension $n=3$, however, conformal flatness is encoded in the Cotton tensor as opposed to the Weyl tensor, and so the following natural definition has been proposed in \cite{ECS3D}: a three-dimensional pseudo-Riemannian manifold is called \emph{conformally symmetric} if its Cotton tensor is parallel, and those which are not conformally flat are then called \emph{ECS} (note that every three-dimensional locally symmetric manifold is conformally flat). There, it is also shown \cite[Theorem 1]{ECS3D} that, reversing the metric if needed, any point in a three-dimensional ECS manifold has a neighborhood isometric to an open subset of
\begin{equation}\label{eqn:model}
(\widehat{M}, \hg) = \big(\R^3, (x^3+\mathfrak{a}(t)x)\,{\rm d}t^2 + {\rm d}t\,{\rm d}s + {\rm d}x^2\big),  
\end{equation} for some suitable smooth function $\mathfrak{a}\colon \R\to \R$. The coordinates $t$ and $s$ of $\widehat{M}$ are called $y$ and $t$ in \cite{ECS3D}, respectively, but have been renamed here as to make \eqref{eqn:model} directly resemble the corresponding local model given in \cite[Section 4]{local_structure} for $n\geq 4$.

The pursuit of compact three-dimensional ECS manifolds quickly comes to an end in view of the following result, interesting on its own right without reference to ECS geometry:

\begin{teoA}\label{teoA}
A compact three-dimensional pseudo-Riemannian manifold with parallel Cotton tensor must be conformally flat.
\end{teoA}

While the compactness assumption here is crucial, Theorem \ref{teoA} may be seen as a close relative (in general signature) of \cite[Theorem 1]{cotton-soliton}: compact Riemannian \emph{Cotton solitons} are conformally flat, but nontrivial compact Lorentzian ones do exist.

\medskip

\noindent {\bf Acknowledgments.} I would like to thank Andrzej Derdzinski for all the comments helping improve the presentation of the text.

\section{Preliminaries}\label{sec:Cotton_3}

Throughout this paper, we work in the smooth category and all manifolds considered are connected. 

\subsection{Symmetries of the Cotton tensor}

The \emph{Cotton tensor} of a $n$-dimensional pseudo-Riemannian manifold $(M,\g)$ is the three-times covariant tensor field ${\rm C}$ on $M$ defined by
\begin{equation}
  \label{eq:cotton}
  {\rm C}(X,Y,Z) = (\nabla_XP)(Y,Z) - (\nabla_YP)(X,Z),\quad\mbox{for } X,Y,Z\in\mathfrak{X}(M).
\end{equation}
Here, $P$ is the \emph{Schouten tensor} of $(M,\g)$, given by
\begin{equation}
  \label{eq:schouten}
  P = {\rm Ric} - \frac{{\rm s}}{2(n-1)}\g,
\end{equation}where ${\rm Ric}$ and ${\rm s}$ stand for the Ricci tensor and scalar curvature of $(M,\g)$, respectively. The Cotton tensor satisfies the following symmetries:
\begin{equation}\label{eqn:Cotton_symmetries}
  \parbox{.6\textwidth}{
    \begin{enumerate}[(i)]
    \item ${\rm C}(X,Y,Z)+{\rm C}(Y,X,Z) = 0$
    \item ${\rm C}(X,Y,Z) + {\rm C}(Y,Z,X) + {\rm C}(Z,X,Y) = 0$
    \item ${\rm tr}_{\g}\big((X,Z) \mapsto {\rm C}(X,Y,Z)\big) = 0$
    \end{enumerate}}
\end{equation}
for all $X,Y,Z\in\mathfrak{X}(M)$. Symmetry (i) is obvious, while (ii) follows from a straightforward computation (six terms cancel in pairs), and (iii) from ${\rm div}\,P = {\rm d}({\rm tr}_{\g}\,P)$ (which, in turn, is a consequence of the twice-contracted differential Bianchi identity ${\rm div}\,{\rm Ric} = {\rm ds}/2$).

\subsection{Algebraic structure in dimension $3$}

A routine computation shows that
\begin{equation}\label{eqn:Ricci_Cotton_model}
  \parbox{.57\textwidth}{the Ricci and Cotton tensors of \eqref{eqn:model} are given by ${\rm Ric} = -3x\,{\rm d}t\otimes {\rm d}t$ and ${\rm C} = 3\,({\rm d}t\wedge {\rm d}x)\otimes {\rm d}t$.}
\end{equation}
The expression for ${\rm C}$ motivates the following result, analogous to \cite[Lemma 17.1]{Tohoku07}:

\begin{teo}\label{lem:olszak_space}
  Let $(V,\langle\cdot,\cdot\rangle)$ be a three-dimensional pseudo-Euclidean space, and ${\rm C}$ be a nonzero \emph{Cotton-like tensor} on $V$, i.e., a three-times covariant tensor on $V$ which formally satisfies \eqref{eqn:Cotton_symmetries}, and consider $\mathcal{D} = \{u\in V \mid {\rm C}(u,\cdot,\cdot) = 0\}$. Then:
  \begin{enumerate}[\normalfont(a)]
  \item $\mathcal{D}$ consists only of null vectors, and hence $\dim \mathcal{D} \leq 1$.
  \item  $\dim \mathcal{D} = 1$ if and only if ${\rm C} = (u\wedge v)\otimes u$ for some $u\in \mathcal{D}\smallsetminus \{0\}$ and unit $v\in \mathcal{D}^\perp$.
  \item In \emph{(b)}, $u$ is unique up to a sign, while $v$ is unique modulo $\mathcal{D}$.
  \end{enumerate}
 Here, we identify $V\cong V^*$ with the aid of $\langle\cdot,\cdot\rangle$.
\end{teo}

\begin{proof}
  For (a), assuming by contradiction the existence of a unit vector \linebreak[4]$e_1\in \mathcal{D}$, we will show that ${\rm C} = 0$. Considering an orthonormal basis $\{e_1,e_2,e_3\}$ for $(V,\langle\cdot,\cdot\rangle)$ and using (\ref{eqn:Cotton_symmetries}-i) and (\ref{eqn:Cotton_symmetries}-ii), we see that
  \begin{equation}\label{eqn:non_essential_C}
    \parbox{.8\textwidth}{${\rm C}_{ijk}$ is only possibly nonzero when $\{i,j,k\} = \{2,3\}$ with $i\neq j$.}
  \end{equation} Now ${\rm C}_{322} = -{\rm C}_{232}$ and ${\rm C}_{323} = -{\rm C}_{233}$, while ${\rm tr}_{\langle\cdot,\cdot\rangle}\big((w,w')\mapsto {\rm C}(e_j,w,w')\big)=0$ for $j=2$ and $j=3$ readily yields ${\rm C}_{233} = 0$ and ${\rm C}_{322}=0$, respectively. Hence ${\rm C} = 0$, as claimed. As for (b), assume that $\dim \mathcal{D} = 1$, fix a null vector $e_1\in \mathcal{D}\smallsetminus \{0\}$, and complete it to a basis $\{e_1,e_2,e_3\}$ of $V$ satisfying the relations
  \begin{equation}\label{eqn:Penrose_frame}
  \langle e_1,e_2\rangle = \langle e_2,e_3\rangle = \langle e_3,e_3\rangle = 0\quad\mbox{and}\quad \langle e_1,e_3\rangle = \langle e_2,e_2\rangle = (-1)^{q+1},
  \end{equation}where $q \in \{1,2\}$ is the index of $\langle\cdot,\cdot\rangle$. By the same argument as in (a), we again obtain \eqref{eqn:non_essential_C}, but this time ${\rm tr}_{\langle\cdot,\cdot\rangle}\big((w,w')\mapsto {\rm C}(e_3,w,w')\big) = 0$ reduces to ${\rm C}_{232} = 0$ in view of \eqref{eqn:Penrose_frame}. Writing $a = {\rm C}_{323} \neq 0$ for the last essential component of ${\rm C}$, it follows that ${\rm C} = a (e^3 \wedge e^2)\otimes e^3$, where $\{e^1,e^2,e^3\}$ is the basis of $V^*$ dual to $\{e_1,e_2,e_3\}$. Applying the isomorphism $V\cong V^*$ and setting $u = |\hspace{.5pt}a|^{1/2}e_1$ and $v= {\rm sgn}(a) e_2$, we obtain the required expression ${\rm C} = (u\wedge v)\otimes u$. Conversely, it is straightforward to verify that the tensor $(u\wedge v)\otimes u$ with $u$ null and $v$ unit and orthogonal to $u$ is Cotton-like with $\mathcal{D}=\R u$ and $\mathcal{D}^\perp = \R u\oplus \R v$. Finally, (c) is clear from (b).
\end{proof}

As a consequence, whenever $(M,\g)$ is a three-dimensional pseudo-Riemannian manifold, we may assign to each point $x\in M$ the kernel $\mathcal{D}_x$ of ${\rm C}_x$ in $(T_xM,\g_x)$. In the ECS case, we have that
\begin{equation}\label{eqn:smooth_distr}
  \parbox{.69\textwidth}{$\mathcal{D}$ is a smooth rank-one parallel distribution on $M$, which contains the image of the Ricci endomorphism of $(M,\g)$.}
\end{equation}
Indeed, we may note that \eqref{eqn:smooth_distr} holds in the model \eqref{eqn:model} (as \eqref{eqn:Ricci_Cotton_model} gives us that $\mathcal{D}$ is spanned by the coordinate vector field $\partial_s$, $\hg$-dual to ${\rm d}t$ up to a factor of $2$), and invoke \cite[Theorem 1]{ECS3D}.

\section{Proof of Theorem A}\label{sec:proof_A}

In this section, we fix a compact three-dimensional ECS manifold $(M,\g)$ and its universal covering manifold $\pi\colon \widetilde{M} \to M$, which equipped with the natural pull-back metric $\tg = \pi^*\g$ becomes an ECS manifold. We will use the same symbols ${\rm Ric}$, $P$, ${\rm C}$, $\nabla$, and $\mathcal{D}$ for the corresponding objects in both $(M,\g)$ and $(\widetilde{M},\tg)$. Observe that
\begin{equation}\label{eqn:compact_quot}
  \parbox{.75\textwidth}{the fundamental group $\Gamma = \pi_1(M)$ acts properly discontinuously on $(\widetilde{M},\tg)$ by deck isometries, with quotient $\widetilde{M}/\Gamma \cong M$.}
\end{equation}
 As $\widetilde{M}$ is simply connected, we may fix two globally defined smooth vector fields $\vec{u}$ and $\vec{v}$ such that ${\rm C} = (\vec{u}\wedge \vec{v})\otimes \vec{u}$ on $\widetilde{M}$. Now, as $\mathcal{D}$ is parallel, item (c) of Theorem \ref{lem:olszak_space} gives us that 
 \begin{equation}\label{eqn:u_parallel}
   \parbox{.9\textwidth}{\begin{enumerate}[(i)] \item $\vec{u}$ is a null parallel vector field spanning $\mathcal{D}$; \item every $\gamma\in \Gamma$ either pushes $\vec{u}$ forward onto itself or onto its opposite.\end{enumerate}}
 \end{equation}
Next, as the Ricci endomorphism of $(\widetilde{M},\tg)$ is self-adjoint, \eqref{eqn:smooth_distr} allows us to write
\begin{equation}\label{eqn:function_f}
  \parbox{.67\textwidth}{${\rm Ric} = -f\,\vec{u}\otimes\vec{u}$, for some smooth function $f\colon \widetilde{M}\to \R$.}
\end{equation}
By \eqref{eqn:function_f} and (\ref{eqn:u_parallel}-i), $(\widetilde{M},\tg)$ is scalar-flat, and so $P={\rm Ric}$. Combining this with (\ref{eqn:u_parallel}-i) again to compute ${\rm C}$ via \eqref{eq:cotton}, we obtain that
\begin{equation}\label{eq:Cotton-gradient}
  \parbox{.65\textwidth}{${\rm C} = (\vec{u}\wedge \nabla f)\otimes\vec{u}$, where $\nabla f$ is the $\tg$-gradient of $f$.}
\end{equation}
However, it follows from (\ref{eqn:u_parallel}-ii) and \eqref{eqn:function_f} that $f$ is $\Gamma$-invariant, and so it has a critical point due to \eqref{eqn:compact_quot} and compactness of $M$. Such a critical point is in fact a zero of ${\rm C}$ by \eqref{eq:Cotton-gradient}, and therefore ${\rm C} = 0$. This is the desired contradiction: $(M,\g)$ must be either noncompact, or conformally flat.

\bibliography{ECS_3D_refs}{}

\begin{thebibliography}{1}

\bibitem{ECS3D}
E.~Calvi\~{n}o Louzao, E.~Garc\'{\i}a-R\'{\i}o, J.~Seoane-Bascoy, and
  R.~V\'{a}zquez-Lorenzo.
\newblock Three-dimensional conformally symmetric manifolds.
\newblock {\em Ann. Mat. Pura Appl. (4)}, 193(6):1661--1670, 2014.

\bibitem{cotton-soliton}
E.~Calvi\~{n}o Louzao, E.~Garc\'{\i}a-R\'{\i}o, and R.~V\'{a}zquez-Lorenzo.
\newblock A note on compact {C}otton solitons.
\newblock {\em Classical Quantum Gravity}, 29(20):205014, 5, 2012.

\bibitem{Chaki-Gupta}
M.~C. Chaki and B.~Gupta.
\newblock On conformally symmetric spaces.
\newblock {\em Indian J. Math.}, 5:113--122, 1963.

\bibitem{TensorNS77}
A.~Derdzi\'{n}ski and W.~Roter.
\newblock On conformally symmetric manifolds with metrics of indices {$0$} and
  {$1$}.
\newblock {\em Tensor (N.S.)}, 31(3):255--259, 1977.

\bibitem{Tohoku07}
A.~Derdzinski and W.~Roter.
\newblock Projectively flat surfaces, null parallel distributions, and
  conformally symmetric manifolds.
\newblock {\em Tohoku Math. J. (2)}, 59(4):565--602, 2007.

\bibitem{local_structure}
A.~Derdzinski and W.~Roter.
\newblock The local structure of conformally symmetric manifolds.
\newblock {\em Bull. Belg. Math. Soc. Simon Stevin}, 16(1):117--128, 2009.

\bibitem{AGAG10}
A.~Derdzinski and W.~Roter.
\newblock Compact pseudo-{R}iemannian manifolds with parallel {W}eyl tensor.
\newblock {\em Ann. Global Anal. Geom.}, 37(1):73--90, 2010.

\bibitem{DT_1}
A.~Derdzinski and I.~Terek.
\newblock New examples of compact {W}eyl-parallel manifolds.
\newblock \emph{Monatsh. Math.} (published online), DOI
  \url{https://doi.org/10.1007/s00605-023-01908-0}.

\bibitem{Roter74}
W.~Roter.
\newblock On conformally symmetric {R}icci-recurrent spaces.
\newblock {\em Colloq. Math.}, 31:87--96, 1974.

\end{thebibliography}
\bibliographystyle{plain}

\end{document}